\documentclass[a4paper,11pt]{article}
\usepackage{footmisc}
\usepackage{geometry}
\usepackage[utf8]{inputenc}
\usepackage[utf8]{inputenc}
\usepackage{amsmath, amsthm, amssymb, amsfonts}
\usepackage{dsfont}
\usepackage[T1]{fontenc}
\usepackage{pst-node}
\usepackage{tikz}
\usetikzlibrary{arrows,matrix,positioning}
\usepackage{tikz-cd} 
\usepackage[english]{babel} 
\usepackage{mathrsfs}
\usepackage[final]{microtype}
\usepackage{url}
\usepackage{setspace}
\usepackage{graphicx}
\usepackage{amsmath}
\usepackage[all]{xy}
\usepackage{enumitem}
\usepackage{tcolorbox}
\usepackage[toc,page]{appendix}
\usepackage{xcolor}
\usepackage{hyperref}
\usepackage{stmaryrd}
\usepackage{upgreek}
\usepackage{mathtools}
\usepackage{filecontents}
\usepackage{multicol}
\usepackage{youngtab}
\usepackage{ytableau}
\usepackage{fancyhdr}
\usepackage{minitoc}
\usepackage{subfiles}
\usepackage{afterpage}

\theoremstyle{plain}
\newtheorem{thm}{Theorem}[section]
\newtheorem{lemma}[thm]{Lemma}
\newtheorem{prop}[thm]{Proposition}
\newtheorem{corollary}[thm]{Corollary}
\theoremstyle{definition}
\newtheorem{defi}[thm]{Definition}

\newtheorem{rmk}[thm]{Remark}
\newtheorem{example}[thm]{Example}

\DeclareMathOperator{\spec}{Spec}
\DeclareMathOperator{\Ima}{Im}
\DeclareMathOperator{\Hom}{Hom}

\DeclareMathOperator{\Lie}{Lie}

\DeclareMathOperator{\End}{End}

\newcommand{\id}{\mathrm{id}}

\newcommand{\G}{\mathbb{G}}
\newcommand{\F}{\mathbb{F}}
\newcommand{\V}{\mathbb{V}}

\newcommand{\A}{\mathbb{A}}

\newcommand{\thistheoremname}{}
\newtheorem*{genericthm*}{\thistheoremname}
\newenvironment{namedthm*}[1]
  {\renewcommand{\thistheoremname}{#1}%
   \begin{genericthm*}}
  {\end{genericthm*}}

\definecolor{bibi}{rgb}{0.79, 0.08, 0.48}
\definecolor{mame}{rgb}{0.0, 0.5, 0.5}

\title{Infinitesimal commutative unipotent group schemes with one-dimensional Lie algebra}
\date{}
\author{Bianca Gouthier\footnote{\noindent\textsc{Max Planck Institute for Mathematics, Vivatsgasse 7, 53111 Bonn, Germany}\\
\textit{Email Address:} \textbf{gouthier@mpim-bonn.mpg.de}}}

\begin{document}

\maketitle

\begin{abstract}
    We prove that over an algebraically closed field of characteristic $p>0$ there are exactly, up to isomorphism, $n$ infinitesimal commutative unipotent $k$-group schemes of order $p^n$ with one-dimensional Lie algebra, and we explicitly describe them. We consequently obtain an explicit description of all infinitesimal subgroup schemes of any supersingular elliptic curve over an algebraically closed field, recovering all their $p^n$-torsions as well. Finally, we use these results to answer a question of Brion on rational actions of infinitesimal commutative unipotent group schemes on curves.
    \end{abstract}

    \noindent\textbf{Acknowledgments.}
This work was done during my PhD at the Institut de Math\'ematiques de Bordeaux carried out under the supervision of Dajano Tossici, to whom I am very grateful. In addition, I thank Michel Brion and Matthieu Romagny for helpful comments and the anonymous referee for their remarks that helped improve this paper.
    
    \section{Introduction}

In general, it is not easy to describe explicitly infinitesimal commutative unipotent group schemes. For example, those that arise as the $p$-torsion of some abelian variety (with $p$-rank zero) are not fully understood and increase in complexity as the dimension grows. What is known is a description in terms of Dieudonné modules (see \cite[V.\textsection1, Theorem 4.3]{DG}), which is somehow more difficult to manipulate. Explicit descriptions of such group schemes, for instance in terms of their Hopf algebra, are useful in order to construct actions of these group schemes on varieties (see, for example, \cite[Proposition 3.17]{gouthier2023infinitesimal} or \cite[\textsection4.1]{Montgomery}).
The main result of this work is Theorem \ref{classification}, which gives an explicit description of all infinitesimal commutative unipotent $k$-group schemes with one-dimensional Lie algebra defined over an algebraically closed field $k$ of characteristic $p>0$, answering partially a question of Fakhruddin (see \cite[Remark 2.10]{Fakhruddin}). In this paper $W$ will denote the group scheme of Witt vectors over $k$, $W_n$ the $k$-group scheme of Witt vectors of length $\leq n$ and $W_n^m$ the kernel of the $m$-th iterate of the Frobenius morphism $F^m\colon W_n\rightarrow W_n$, for any $n,m\geq1$. 

\begin{thm}\label{classification}
Let $k$ be an algebraically closed field of characteristic $p>0$. For any $n\geq 1$, there are exactly, up to isomorphism, $n$ infinitesimal commutative unipotent $k$-group schemes of order $p^n$ and with one-dimensional Lie algebra. They are the group schemes of the form $$W_n^n[V-F^i]:= \ker(V-F^i\colon W_n^n\rightarrow W_n^n)$$ for some $i=1,\dots,n$. 
\end{thm}

This result was known for infinitesimal commutative unipotent group schemes of order $\leq p^3$ thanks to \cite[(15.5)]{Oort} and \cite[Theorem 1.1]{NWW}.

Of these group schemes, the only ones that are contained in a smooth connected one-dimensional algebraic group are $\alpha_{p^n}$ and $W_n^n[F-V]$ (the former is a subgroup of $\G_a$ and the latter is the $n$-th Frobenius kernel of a supersingular elliptic curve - see Proposition \ref{ellcurves}) for any $n\geq1$, see Proposition \ref{guciinellcur}. All others are examples of infinitesimal group schemes that act rationally and generically freely on any curve (by \cite[Theorem 1.4]{gouthier2023infinitesimal}), but are not subgroup schemes of a smooth connected one-dimensional algebraic group. We answer in this way to a question of Brion (see \cite[\textsection1]{brion2022actions}). Notice moreover that $W_n^n[F-V]$ is the only self-dual group scheme of the list. If one considers infinitesimal commutative unipotent group schemes with higher dimensional Lie algebra, this is not the case anymore: indeed the $p$-torsion of any principally polarized abelian variety of dimension $g$ and $p$-rank zero, is a self-dual infinitesimal commutative unipotent group scheme, and there exist $p^{g-1}$ different isomorphism classes of such group schemes (see \cite{pries2006short}). 

\vspace{1em}
\noindent\textbf{Outline of the paper.} We begin with some recalls on finite group schemes as well as on Witt vectors and Dieudonn\'e modules. After some preliminary results we proceed with the proof of Theorem \ref{classification}. In the last part we obtain, as consequences of Theorem \ref{classification}, an explicit description of all infinitesimal subgroup schemes of any supersingular elliptic curve (included their $p^n$-torsions) over an algebraically closed field (Proposition \ref{ellcurves}) and the answer to a question of Brion on rational actions of infinitesimal commutative unipotent group schemes on curves (Proposition \ref{guciinellcur}). We end with an example of a non-commutative infinitesimal unipotent group scheme with one dimensional Lie algebra, answering to similar questions raised by both Fakhruddin and Brion (see \cite[Remark 2.10]{Fakhruddin} and \cite[\textsection1]{brion2022actions}).

\section{Preliminaries}

Throughout the work, $k$ will denote a perfect field of characteristic $p>0$. Let us recall here the definitions of the principal objects of this paper, together with some results concerning them. The main background references for this are \cite{DG}, \cite{Milne} and \cite{Pink}.

\subsection{Finite group schemes}

For $G$ a $k$-group scheme, $G^{(p)}$ denotes the Frobenius twist of $G$, that is the base change $G\times_{k,f}\spec(k)$ with respect to the Frobenius morphism $f\colon k\rightarrow k,c\mapsto c^p$ of $k$. Iteratively, one defines $G^{(p^n)}$. Recall that the \textit{height} of an infinitesimal group scheme $G$ is the smallest integer $n>0$ such that the $n$-th Frobenius morphism $F_G^n:G\rightarrow G^{(p^n)}$ vanishes. Moreover, we denote by $k[G]$ its Hopf algebra. 

\begin{defi}[Lie algebra]
    Let $G$ be an affine $k$-group scheme and denote by $I_G=\ker(\varepsilon)$ its augmentation ideal (where $\varepsilon$ is the counit map $\varepsilon\colon k[G]\rightarrow k$). We define the \textit{Lie algebra} of $G$ to be $\Lie (G)=\Hom_k(I_G/I_G^2,k).$ As a $k$-vector space $\Lie (G)$ is the tangent space of $G$ at the identity element $e_G$ and it has an additional structure of Lie algebra (see, for example, \cite[II.\textsection4, 4]{DG}).
\end{defi}

\begin{defi}[Infinitesimal group scheme]
    A $k$-group scheme $G=\spec(A)$ is said to be \textit{infinitesimal} if its augmentation ideal $I_G=\ker(\varepsilon\colon A\rightarrow k)$ is nilpotent.
\end{defi}

Non-trivial infinitesimal group schemes exist only over fields of positive characteristic: indeed, by Cartier's Theorem, in characteristic zero all algebraic groups are smooth.
Moreover, infinitesimal $k$-group schemes are group schemes that topologically are just a point: indeed $I_G$ is nilpotent if and only if the topological spaces $|\spec A|$ and $e_G=|\spec A/I_G|$ are isomorphic (and we always have $A/I_G\simeq k$).
The structure of the underlying scheme of an infinitesimal group scheme over a perfect field is well-known: we recall it in the following Theorem.

\begin{thm}\label{infgrpstrctr}
    Let $k$ be a perfect field of characteristic $p>0$ and $G$ be an infinitesimal $k$-group scheme. Then $$k[G]\simeq k[T_1,\dots,T_s]/(T_1^{p^{e_1}},\dots,T_s^{p^{e_s}})$$ for some integers $e_1,\dots,e_s\geq1.$ In particular $s=\dim_k(\Lie(G)).$
\end{thm}

\begin{proof}
    See \cite[Theorem 11.29]{Milne}.
\end{proof}

\begin{defi}[Unipotent group scheme]
    A $k$-algebraic group $G$ is said to be \textit{unipotent} if it is isomorphic to an algebraic subgroup of the $k$-algebraic group of upper triangular unipotent matrices $U_n$ for some $n\geq1$.
\end{defi}

Recall that a \textit{finite $k$-group scheme} is a $k$-group scheme that is finite as a $k$-scheme and that the category of finite commutative group schemes over a field $k$ is abelian. We call \textit{order} $o(G)$ of a finite $k$-group scheme $G=\spec(A)$ the dimension of $A$ as a $k$-vector space. 

\begin{lemma}
    Let $G=\spec(A)$ be a finite (commutative) $k$-group scheme. Then the dual of $A$ as a $k$-vector space $$A^\vee=\Hom_k(A,k)$$ is a finite dimensional (commutative) $k$-Hopf algebra.
\end{lemma}

\begin{proof}
    See \cite[V.\textsection1.2.10]{DG}. We just recall here that the multiplication and comultiplication of $A^\vee$ are respectively given by the convolution product $$m^\vee\colon A^\vee\otimes_kA^\vee\rightarrow A^\vee, \quad
        f\otimes g\mapsto (A\stackrel{\Delta}{\rightarrow}A\otimes_kA\stackrel{f\otimes g}{\rightarrow}k\otimes_kk\simeq k)$$ and $$
        \Delta^\vee\colon A^\vee\rightarrow A^\vee\otimes_kA^\vee,\quad\varphi\mapsto(A\otimes_kA\stackrel{m}{\rightarrow}A\stackrel{\varphi}{\rightarrow}k).
    $$
\end{proof}

\begin{defi}[Cartier dual] Let $G$ be a finite commutative $k$-group scheme.
    We call \textit{Cartier dual} of $G$ the finite commutative $k$-group scheme $$G^\vee=\spec(A^\vee).$$
\end{defi}

\begin{rmk}\label{rmk: Frob and Ver duality}
The Cartier dual gives an exact contravariant functor from the category of finite commutative $k$-group schemes to itself.
    When $G$ is a finite commutative $k$-group scheme one can verify that the Verschiebung morphism $V_G\colon G^{(p)}\rightarrow G$ coincides with the dual of the Frobenius morphism $F_{G^\vee}\colon{G^\vee} \rightarrow(G^\vee)^{(p)}\simeq(G^{(p)})^\vee $ of the Cartier dual $G^\vee$. Moreover it holds $V_G\circ F_G=p_G$ and $F_G\circ V_G=p_{G^{(p)}}$ (see \cite[IV.\textsection3, 4.9-10]{DG} for these two facts).
\end{rmk}

\begin{prop}\label{prop: characterization finite grps}
    Let $G$ be a finite commutative $k$-group scheme.
    \begin{multicols}{2}
\begin{enumerate}
    \item The following are equivalent:
    \begin{enumerate}
        \item $G$ is étale,
        \item $F_G$ is an isomorphism,
        \item $V_{G^\vee}$ is an isomorphism,
        \item $G^\vee$ is of multiplicative type.
    \end{enumerate}

\columnbreak
    \item The following are equivalent:
    \begin{enumerate}
        \item $G$ is infinitesimal,
        \item $F_G$ is nilpotent,
        \item $V_{G^\vee}$ is nilpotent,
        \item $G^\vee$ is unipotent.
    \end{enumerate}
\end{enumerate}
\end{multicols}
\end{prop}

\begin{proof}
    See \cite[IV.\textsection3, 5.3]{DG}.
\end{proof}

The following is a useful lemma on infinitesimal group schemes. 

\begin{lemma}\label{infgrps}Let $G$ be an infinitesimal $k$-group scheme of order $p^n$ for some $n\geq 0$. Then: \begin{enumerate}[label=\roman*)]     \item $\max\left(\dim_k(\Lie(G)),\mathrm{ht}(G)\right)\leq n$;   \item $n\leq\dim_k(\Lie(G))\times\mathrm{ht}(G)\leq n\times\min\left(\dim_k(\Lie(G)),\mathrm{ht}(G)\right)$; 
\item if $G$ is also commutative and unipotent then also $V_G^n=0$, where $V_G$ is the Verschiebung morphism. \end{enumerate}\end{lemma}

\begin{proof}We can suppose that $k$ is perfect. Then by Theorem \ref{infgrpstrctr} we have $$k[G]\simeq k[T_1,\dots,T_s]/(T_1^{p^{e_1}},\dots,T_s^{p^{e_s}})$$ as $k$-algebras where $$1\leq e_1\leq \dots\leq e_s, \quad e_1+\dots+e_s=n,\quad e_s=\mathrm{ht}(G)\quad \mbox{and} \quad s=\dim_k(\Lie(G)).$$ The first two points then follow, indeed \begin{enumerate}[label=\roman*)] \item $n=e_1+\dots+e_s\geq s=\dim_k(\Lie(G))$ and $n=e_1+\dots+e_s\geq e_s=\mathrm{ht}(G)$ yielding that $n\geq \max(\dim_k(\Lie(G)),\mathrm{ht}(G))$; and \item $n=e_1+\dots+e_s\leq s\times e_s=\dim_k(\Lie(G))\times\mathrm{ht}(G)\leq n\times\min\left(\dim_k(\Lie(G)),\mathrm{ht}(G)\right)$ where the second inequality follows from the first point. \item If $G$ is infinitesimal commutative unipotent, then its dual $G^\vee$ is also such (see Proposition \ref{prop: characterization finite grps}) and has order $p^n$. Then, applying the first statement, we have $F_{G^\vee}^n=0$ and thus $V_G^n=0$ (by Remark \ref{rmk: Frob and Ver duality}).\end{enumerate}
 
\end{proof}

\begin{rmk}\label{onenrmk}Notice that, by the second statement of the above Lemma, we have in particular that $\dim_k(\Lie(G))=1$ if and only if $\mathrm{ht}(G)=n$ and  that $\mathrm{ht}(G)=1$ if and only if $\dim_k(\Lie(G))=n$.\end{rmk}

\subsection{Witt vectors and Dieudonn\'e modules}
 
The group scheme of Witt vectors $W$ over a perfect field $k$ of positive characteristic $p$ plays a central role in the study of unipotent commutative $k$-group schemes. We thus recall here some of its main properties that will be used freely later on. A reference for this is \cite[V.\textsection1 and \textsection4]{DG}. As a $k$-scheme, $W$ coincides with $\A_k^{\mathbb{N}}$ and it is endowed with a structure of ring scheme coming from Witt polynomials. We will mostly be interested in its structure as a group scheme.
We denote by $W_n$ the $k$-subgroup scheme of $W$ of Witt vectors of length $\leq n$ and by $W_n^m$ the kernel of $F^m\colon W_n\rightarrow W_n$ (where $F$ is the Frobenius morphism of $W_n$). As a $k$-scheme, $W_n$ coincides with $\A_k^n$ and we denote by $k[T_0,\dots,T_{n-1}]$ its $k$-Hopf algebra. Notice that if we want to consider $r$ copies of $W_n$ we will use the notation $(W_n)^r$ with the parenthesis. The $k$-group scheme $W_n^m$ is the Cartier dual of $W_m^n$ for every $n,m\geq1.$ 
If for example we consider $(x_0,x_1)$ and $(y_0,y_1)$ two Witt vectors of length $2$, then their sum is given by $$(x_0,x_1)+(y_0,y_1)=\left(x_0+y_0,x_1+y_1+S_1(x_0,y_0)\right)$$ where $S_1(x_0,y_0)=\frac{x_0^p+y_0^p-(x_0+y_0)^p}{p}=-\sum_{k=1}^{p-1}\frac{1}{p}\binom{p}{k}x_0^ky_0^{p-k}$.
In general, the expression of the sum of Witt vectors becomes more complex as the length increases. One important property is that, if we give weight $p^j$ to the $j$-th coordinate for any $j\in\mathbb{N}$, then the polynomial expressing the $i$-th term of the sum is homogeneous of degree $p^i$; moreover it involves just the coordinates of the vectors up to the index $i$.
Notice that $W$ is defined over $\F_p$ and thus the Frobenius twist $W^{(p)}$ is isomorphic to $W$.
Let $V$ be the Verschiebung morphism of $W_n$. On points $F$ and $V$ act as follows: $$F\left((x_0,x_1,\dots,x_{n-1})\right)=(x_0^p,x_1^p,\dots,x_{n-1}^p) $$ and $$ V\left((x_0,x_1,\dots,x_{n-1})\right)=(0,x_0,x_1,\dots,x_{n-2}).$$
Moreover $F\circ V=V\circ F=p\cdot\id$, therefore $$p\cdot (x_0,x_1,\dots,x_{n-1})=(0,x_0^p,x_1^p,\dots,x_{n-2}^p).$$ 

Let us recall the following important fact, that allows us, when studying infinitesimal commutative unipotent group schemes over a perfect field, to work with Witt vectors.

\begin{prop}\label{infwitt}
If $k$ is perfect, then every infinitesimal commutative unipotent $k$-group scheme $G$ can be embedded in $\left(W_n^m\right)^r$ for some $n,m,r\geq1$.
\end{prop}

\begin{proof}
   See for example \cite[V.\textsection1, Proposition 2.5]{DG} or \cite[Proposition 22.5]{Pink}.
\end{proof}

\begin{rmk}\label{FVnilp}
A finite commutative $k$-group scheme is infinitesimal unipotent if and only if its Frobenius and Verschiebung morphisms are both nilpotent (see Proposition \ref{prop: characterization finite grps}). In particular, in Proposition \ref{infwitt} one can take $m$ and $n$ to be respectively their nilpotency indices (this is a direct consequence of the functoriality of the Frobenius and Verschiebung morphism). 
\end{rmk}

The collection of $W_n^m$ becomes a direct system with index set $\mathbb{N}\times\mathbb{N}$ via the homomorphisms
\begin{center}
    \begin{tikzcd}
        W_n^m\arrow[d,hook,"v"]\arrow[r,hook,"i"]&W_n^{m+1}\arrow[d,hook,"v"]\\
        W_{n+1}^m\arrow[r,hook,"i"]&W_{n+1}^{m+1}
    \end{tikzcd}
\end{center} where $v$ and $i$ are the monomorphisms induced naturally by the Verschiebung and the inclusion respectively. Let $\sigma\colon W(k)\rightarrow W(k)$ be the ring endomorphism induced by the Frobenius morphism $F$.

\begin{defi}
    We will denote by $E$ the ring of  non-commutative polynomials over the ring $W(k)$ in two variables $\F$ and $\V$ subject to the following relations:
    \begin{itemize}
        \item $\F\cdot\xi=\sigma(\xi)\cdot \F$ for all $\xi\in W(k)$;
        \item $\V\cdot\sigma(\xi)=\xi\cdot \V$ for all $\xi\in W(k)$;
        \item $\F \V=\V\F=p$.
    \end{itemize}
\end{defi}

Notice that $E$ is a free left (or right) $W(k)$-module with basis $\{\dots, \V^2, \V, 1, \F, \F^2,\dots\}$ (see \cite[V.\textsection1,3]{DG}).

\begin{prop}
    There exists a unique ring homomorphism $E\rightarrow\End(W_n^m)$ for all $n,m$ such that $\F$ and $\V$ act as the Frobenius and the Verschiebung morphisms and $\xi\in W(k)$ acts through multiplication by $\sigma^{-n}(\xi).$ These actions of $E$ are compatible with the homomorphisms $i$ and $v$ of the directed system.
\end{prop}

\begin{proof}
    See \cite[V.\textsection1, 3.4]{DG}.
\end{proof}

\begin{defi}[Dieudonn\'e module]
For any infinitesimal commutative unipotent $k$-group scheme, we define $$M(G):=\varinjlim_{n,m}\Hom(G,W_n^m)$$ with left $E$-module structure given by the action of $E$ on $W_n^m$, called \textit{Dieudonn\'e module of $G$}. 
    Then $M$ is a left exact additive functor from the category of infinitesimal commutative unipotent $k$-group schemes to that of left $E$-modules (see for example \cite[Lecture 10]{Pink}).
\end{defi}

For $G$ an infinitesimal commutative unipotent $k$-group scheme, $\mathrm{length}_{W(k)}\left(M(G)\right)=\mathrm{log}_p\left(o(G)\right)$ and the functor $M$ is actually an exact functor; see, for example, \cite[Proposition 23.3 and Lemma 23.5]{Pink}.

\begin{thm}
    The functor $M$ induces an anti-equivalence of categories between the category of infinitesimal commutative unipotent $k$-group schemes to that of left $E$-modules of finite length with $F$ and $V$ nilpotent. 
\end{thm}

\begin{proof}
See \cite[V.\textsection1, Theorem 4.3]{DG} or \cite[Theorem 23.2]{Pink}.
\end{proof}

\section{Around group schemes of the form $W_n^{n'}[F^r-V^s]$}

From now on, for all $n,n',r,s\geq1$ integers, $W_n^{n'}[F^r-V^s]$ will denote the kernel of the morphism $F^r-V^s\colon W_n^{n'}\rightarrow W_n^{n'}$.
In the following Lemma we compute the Hopf algebra of certain infinitesimal commutative unipotent group schemes of this form. The results of this section are not directly related to the proof of the main Theorem of this work. Nevertheless, they could play a key role for the study of infinitesimal commutative unipotent group schemes with Lie algebra of dimension greater than one and contained in just one copy of Witt vectors. For this reason, together with the fact that they are interesting on their own, it was decided to include them.

\begin{lemma}\label{Hopflagbras} 
For every integer $r,s\geq1$ and $m\geq2$ let $d=lcm(r,s)$ and consider the $k$-group scheme
    $G=W_{md}^{md}[F^r-V^s].$ Then
    \begin{enumerate}[label=(\roman*)]
        \item $G=W_{n}^{n'}[F^r-V^s]$ where $n=\min\left(sm\frac{d}{r},md\right)$ and $n'=\min\left(rm\frac{d}{s},md\right)$, and
        \item the $k$-Hopf algebra of $G$ is $$k[G]=k[T_0,\dots,T_{n-1}]/(T_0^{p^r},\dots,T_{s-1}^{p^r},T_s^{p^r}-T_0,\dots,T_{n-1}^{p^r}-T_{n-s-1})$$ where $k[T_0,\dots,T_{n-1}]$ is the $k$-Hopf algebra $k[W_n]$ of Witt vectors of length $\leq n$, that is the comultiplication on $k[G]$ is given by that of Witt vectors.
    \end{enumerate}
\end{lemma}

\begin{proof}
For the first statement, the inclusion $W_{n}^{n'}[F^r-V^s]\subseteq G$ is clear since $n,n'\leq md$. For the other inclusion, observe that $s\mid n$ and $r\mid n'$ and thus $$V_G^n=(V_G^s)^{\frac{n}{s}}=(F_G^r)^{\frac{n}{s}}=0\quad\mbox{and}\quad F_G^{n'}=(F_G^r)^{\frac{n'}{r}}=(V_G^s)^{\frac{n'}{r}}=0.$$
    For the second statement, for any $k$-algebra $R$ we have $$G(R)=$$$$\left\{\underline{a}\in W_n(R)\mid \left(a_0^{p^r},\dots,a_{s-1}^{p^r},a_s^{p^r},\dots,a_{n-1}^{p^r}\right)=\left(0,\dots,0,a_0,\dots,a_{n-s-1}\right)\mbox{ and } a_i^{p^{n'}}=0\;\forall i\right\}$$$$=\left\{\underline{a}\in W_n(R)\mid \left(a_0^{p^r},\dots,a_{s-1}^{p^r},a_s^{p^r}-a_0,\dots,a_{n-1}^{p^r}-a_{n-s-1}\right)=\underline{0}\mbox{ and } a_i^{p^{n'}}=0\;\forall i\right\}$$$$=\left\{\underline{a}\in W_n(R)\mid \left(a_0^{p^r},\dots,a_{s-1}^{p^r},a_s^{p^r}-a_0,\dots,a_{n-1}^{p^r}-a_{n-s-1}\right)=\underline{0}\right\}$$ where $\underline{a}=(a_0,\dots,a_{n-1})$ and the last equality is due to the fact that $n'\geq 2r$: indeed clearly $a_i^{p^{n'}}=0$ for every $i=0,\dots,s-1$ since in this case $a_i^{p^{r}}=0$, and for $i=s,\dots,n-1$ we have $a_i^{p^{n'}}=a_{i-s}^{p^{n'-r}}=0.$
\end{proof}

\begin{example}
Notice that in general the Hopf algebra of a group scheme of the form $W_n^n[F^r-V^s]$ for any $n\geq1$ and $r,s=1,\dots,n$ does not have the "regular" shape appearing for the cases treated in Lemma \ref{Hopflagbras}.
Take for example $$G=W_3^3[V-F^2]=W_2^3[V-F^2].$$ Notice that $G=\ker(F_H^3)$ where $H=W_4^4[V-F^2]=W_2^4[V-F^2]=\spec(A)$ and $$A=k[T_0,T_1]/(T_0^{p^2},T_1^{p^2}-T_0)$$ by Lemma \ref{Hopflagbras}. Therefore, $$k[G]=A/(T_0^{p^3},T_1^{p^3})=k[T_0,T_1]/(T_0^{p^2},T_1^{p^2}-T_0,T_0^{p^3},T_1^{p^3})=k[T_0,T_1]/(T_0^{p},T_1^{p^2}-T_0).$$ \end{example}

As recalled in Proposition \ref{infwitt}, if the base field $k$ is perfect, every infinitesimal commutative unipotent $k$-group scheme $G$ can be embedded in $\left(W_n^m\right)^r$ for some $n,m,r\geq1$. When either $G$ or its Cartier dual has one-dimensional Lie algebra, then they both embed in just one copy of the group scheme of Witt vectors (see \cite[Lemma 2.28]{gouthier2023infinitesimal}), but this is not the only case. Below we see that there are many examples of infinitesimal commutative unipotent group schemes $G$ that embed in just one copy of the group scheme of Witt vectors and such that neither $G$ nor $G^\vee$ have one-dimensional Lie algebra.

\begin{prop}\label{duality}
 For every integer $r,s\geq1$ and $m\geq2$ let $d=lcm(r,s)$, $n=\min\left(sm\frac{d}{r},md\right)$ and $n'=\min\left(rm\frac{d}{s},md\right)$. The Dieudonn\'e module of the $k$-group scheme
    $$G=W_{md}^{md}[F^r-V^s]=W_{n}^{n'}[F^r-V^s]$$ is $$M(G)=E/(E(\F^r-\V^s)+E\F^{n'}).$$ Moreover, the Cartier dual of $G$ is $$G^\vee=W_{n'}^n[V^r-F^s].$$ 
\end{prop}

\begin{proof}
    Recall (see, for example, \cite[Proposition 23.3]{Pink}) that $E_n^{n'}:= M(W_n^{n'})=E/(E\F^{n'}+E\V^n).$ The short exact sequence $$0\rightarrow G\rightarrow
     W_n^{n'}\rightarrow \Ima(F^r-V^s)\rightarrow0$$ yields the short exact sequence of $E$-modules $$0\rightarrow E_n^{n'}(\F^r-\V^s)\rightarrow E_n^{n'}\rightarrow M(G)\rightarrow0$$ and thus $M(G)=E/(E(\F^r-\V^s)+E\F^{n'})$ as stated.
     
     Consider now the Cartier dual $G^\vee$ of $G$. First of all, let us show that $G^\vee$ embeds in just one copy of the $k$-group scheme of Witt vectors. Indeed, if this was not the case, by \cite[Lemma 2.28]{gouthier2023infinitesimal}, $G^\vee$ would contain $\alpha_p\times_k\alpha_p$ as a $k$-subgroup scheme. As a consequence, we would have a surjection $$G\twoheadrightarrow\alpha_p\times_k\alpha_p,$$ implying that $k[\alpha_p\times_k\alpha_p]=k[U_1,U_2]/(U_1^p,U_2^p)$ is a $k$-Hopf subalgebra of $k[G]$. This is not the case since, up to scalar multiplication, there exists a unique element $x\in k[G]$ such that $\Delta(x)=x\otimes1+1\otimes x$ and $x^p=0$ (this element is $T_0^{p^{r-1}}$ with the notation of the description of $k[G]$ given in Lemma \ref{Hopflagbras}), while $k[\alpha_p\times_k\alpha_p]$ has two $k$-linearly independent elements with this property. Therefore $G^\vee\subseteq W_{n'}^n$. Moreover, $V^r_{(G^\vee)^{(1/p)}}-F^s_{G^\vee}=(F^r_G-V^s_{G^{(1/p)}})^\vee=0$ and thus $$G^\vee\subseteq W_{n'}^n[V^r-F^s].$$ By Lemma \ref{Hopflagbras}, $$o\left(W_{n'}^n[V^r-F^s]\right)=p^{sn'}=p^{rn}=o(G)=o(G^\vee),$$ hence the equality.
\end{proof}

\begin{example}
Notice that in general it is not true that an infinitesimal commutative unipotent $k$-group scheme $H$ and its Cartier dual $H^\vee$ are contained in the same (minimal) number of copies of Witt vectors. More precisely, by Proposition \ref{infwitt} we know that there exist embeddings $H\subseteq \left(W_n^m\right)^r$ and $H^\vee\subseteq \left(W_{n'}^{m'}\right)^{r'}$. Take $n,m,r,n',m',r'$ minimal realizing the embeddings. By Cartier duality one has $n'=m$ and $m'=n$. Nevertheless, in general $r\neq r'$. Let us give an example of such a group scheme $H$. 
Consider the $k$-subgroup scheme of $W_2\times_kW_2$ given by $$G=\spec\left(k[T_0,T_1,U_0,U_1]/(T_0^p,T_1^p-U_0,U_0^p,U_1^p-T_0)\right)$$ and take $H$ to be the quotient $$G\twoheadrightarrow H=\spec\left(k[T_0,T_1,U_0]/(T_0^p,T_1^p-U_0,U_0^p)\right)\simeq \spec\left(k[T_0,T_1]/(T_0^p,T_1^{p^2})\right).$$ Then $H$ is a $k$-subgroup scheme of $W_2$, that is $r=1$. Let us show that $r'=2$, that is its dual cannot be embedded in just one copy of Witt vectors. 
    One can see that $G$ is self-dual setting $\widetilde{T}_0=T_1^*, \widetilde{T}_1=U_0^*,\widetilde{U}_0=U_1^*,\widetilde{U}_1=T_0^*$ and, as a consequence, we have $$H^\vee=\spec\left(k[\widetilde{T}_0,\widetilde{T}_1,\widetilde{U}_1]/(\widetilde{T}_0^p,\widetilde{T}_1^p,\widetilde{U}_1^p-\widetilde{T}_0)\right)\simeq\spec\left(k[X,Y]/(X^{p^2},Y^p)\right)$$ with comultiplication \begin{align*}
        \Delta\colon  X&\mapsto X\otimes1+1\otimes X,\\
        Y&\mapsto Y\otimes1+1\otimes Y+S_1(X^p\otimes1,1\otimes X^p)
    \end{align*} where $S_1(a,b)=\frac{a^p+b^p-(a+b)^p}{p}=-\sum_{k=1}^{p-1}\frac{1}{p}\binom{p}{k}a^kb^{p-k}$. Notice that $\ker(F_{H^\vee})\simeq\alpha_p\times_k\alpha_p$ and, by \cite[Lemma 2.28]{gouthier2023infinitesimal}, one deduces that $r'=2$.

    Let us conclude the example with a remark on the group scheme $G$:
    one can actually show that $G$ is isomorphic to $$\ker(F-V\colon W_2\rightarrow W_2)^2=\spec\left(k[X_0,X_1,Y_0,Y_1]/(X_0^p,X_1^p-X_0,Y_0^p,Y_1^p-Y_0)\right).$$ The isomorphism is explicitly given by \begin{align*}
        X_0&\mapsto T_0+U_0,\\
        X_1&\mapsto T_1+U_1+S_1(T_0,U_0),\\
        Y_1&\mapsto  -T_1+U_1+S_1(T_0,-U_0),\\
        Y_0&\mapsto T_0-U_0.
    \end{align*} As shown later on (see Proposition \ref{ellcurves}), over an algebraically closed field, this is the product of two copies of the $p$-torsion of a supersingular elliptic curve over $k$.
\end{example}

\section{Proof of the main theorem}

This section is devoted to the proof of Theorem \ref{classification}, which is given after two preliminary results.

\begin{lemma}\label{induction}
    Let $G$ be an infinitesimal commutative unipotent $k$-group scheme of order $p^n$ with one-dimensional Lie algebra. Then, up to (canonical) isomorphism, the Dieudonn\'e module of $\ker(F_G^{n-1})$ is given by the quotient $M(G)/M(G)\F^{n-1}.$ 
\end{lemma}

\begin{proof}
We start by observing that, by Remark \ref{onenrmk}, $G$ has height $n$. Moreover, by Theorem \ref{infgrpstrctr}, $G\simeq\spec(k[T]/(T^{p^n}))$ as $k$-schemes and thus $$F_G^{n-1}(G)\simeq\spec(k[T^{p^{n-1}}]/(T^{p^n}))\simeq\spec(k[U]/(U^p)).$$ Since $F_G^{n-1}(G)$ is still unipotent, then $F_G^{n-1}(G)\simeq\alpha_p.$
    Therefore, we have the short exact sequence $$0\longrightarrow H\longrightarrow G\stackrel{F_G^{n-1}}{\longrightarrow}\alpha_p\longrightarrow0$$ where $H:=\ker(F_G^{n-1}).$ 
    Applying the (exact contravariant) Dieudonn\'e functor we obtain the short exact sequence $$0\longrightarrow M(\alpha_p)\stackrel{M(F_G^{n-1})}{\longrightarrow} M(G)\longrightarrow M(H)\longrightarrow0.$$ Now, $\F^{n-1}$ is zero in $M(H)$ and thus we have the factorization $$M(G)/M(G)\F^{n-1}\twoheadrightarrow M(H).$$ Therefore $M(G)\F^{n-1}$ is contained in the kernel of $M(G)\rightarrow M(H)$ that is (isomorphic to) $M(\alpha_p).$ Finally, since $M(G)\F^{n-1}\neq0$ and $M(\alpha_p)$ has length one, then also $M(G)\F^{n-1}$ has length one (and $M(G)\F^{n-1}\simeq M(\alpha_p)$). Therefore, $M(G)/M(G)\F^{n-1}$ is isomorphic to $M(H)$. 
\end{proof}

Let us recall the definition of \textit{Teichmüller lift} of an element of the field $k$.

\begin{defi}[Teichmüller lift]
    Let $a_0$ be an element of $k$. Then, the Teichmüller lift of $a_0$ is the element $[a_0]=(a_0,0,0,\dots)$ in the ring $W(k)$ of Witt vectors.
\end{defi}

Notice that if $a$ is the Teichmüller lift of $a_0$, then $a_0=a\mod p$.

\begin{prop}\label{Emodulesisom}
    Let $k$ be algebraically closed and $n\geq3$. Then for every $a\in W(k)$ and $i=1,\dots,n-2$ we have an isomorphism of $E$-modules $$E/\left(E(\V-\F^i-a\F^{n-1})+E\F^n\right)\simeq E/\left(E(\V-\F^i)+E\F^n\right).$$ Moreover, for every $a\in W(k)^\times$ it holds $$E/\left(E(\V-a\F^{n-1})+E\F^n\right)\simeq E/\left(E(\V-\F^{n-1})+E\F^n\right)$$ as $E$-modules.
\end{prop}

\begin{proof}
 We will show that for any $i=1,\dots,n-2$
    \begin{align*}
       \varphi\colon  E/\left(E(\V-\F^i-a\F^{n-1})+E\F^n\right)&\rightarrow E/\left(E(\V-\F^i)+E\F^n\right)\\
        1&\mapsto 1+c\F^{n-1-i}
    \end{align*}
    is an isomorphism for a good choice of $c\in W(k).$
    First, we have to find under what conditions on $c$ is $\varphi$ well-defined. This is the case if and only if $\varphi(\V-\F^i-a\F^{n-1})=0.$ Now, $$\varphi(\V-\F^i-a\F^{n-1})=(\V-\F^i-a\F^{n-1})(1+c\F^{n-1-i})=$$$$\V-\F^i-a\F^{n-1}+(\V-\F^i-a\F^{n-1})c\F^{n-1-i}=-a\F^{n-1}+\sigma^{-1}(c)\V\F^{n-1-i}-\sigma^i(c)\F^{n-1}=$$$$-a\F^{n-1}+\sigma^{-1}(c)\F^{n-1}-\sigma^i(c)\F^{n-1}=(-a+\sigma^{-1}(c)-\sigma^i(c))\F^{n-1}$$ where $\sigma$ is the Frobenius morphism on $W(k).$ Let us denote $\gamma:= -a+\sigma^{-1}(c)-\sigma^i(c).$ Since $\F^{n}=0$, it is enough to find $c$ such that $\gamma=0\mod p$. We take $c=[c_0]$ to be the Teichm\"uller lift of a non-zero solution $c_0$ of the polynomial $X-X^{p^{i+1}}=a_0^p$, which exists since $k$ is algebraically closed. Therefore, the morphism $\varphi$ is well-defined. Let us show that $\varphi$ is injective. Since in the source of $\varphi$ we have $\V=\F^i+a\F^{n-1}$, a general element of the source is of the form $\sum_{j=0}^{n-1}b_j\F^j$ with $b_j\in W(k)$. Moreover, since $k$ is perfect, for any $\alpha=(\alpha_0,\alpha_1,\dots)\in W(k)$ we have $$\alpha=[\alpha_0]+(0,\alpha_1,\alpha_2,\dots)=[\alpha_0]+p\left(\alpha_1^{1/p},\alpha_2^{1/p},\dots\right)=[\alpha_0]+\left(\alpha_1^{1/p},\alpha_2^{1/p},\dots\right)p$$ and thus $$\alpha\F^j=[\alpha_0]\F^j+\left(\alpha_1^{1/p},\alpha_2^{1/p},\dots\right)\V\F^{j+1}=[\alpha_0]\F^j+\left(\alpha_1^{1/p},\alpha_2^{1/p},\dots\right)\F^{j+1+i}.$$ Since $\F^n=0$, repeating the argument for $\left(\alpha_1^{1/p},\alpha_2^{1/p},\dots\right)\F^{j+1+i}$ we see that for a general element $\sum_{j=0}^{n-1}b_j\F^j$ in the quotient we can suppose that $b_j$ is the Teichm\"uller lift of some element of the base field for any $j=0,\dots,n-1$. Suppose now that such an element maps to zero. Then we have $$0=\left(\sum_{j=0}^{n-1}b_j\F^j\right)(1+c\F^{n-1-i})=\sum_{j=0}^{n-1}b_j\F^j+\sum_{j=0}^{n-1}b_j\F^jc\F^{n-1-i}=$$$$\sum_{j=0}^{n-1}b_j\F^j+\sum_{j=0}^{n-1}b_j\sigma^j(c)\F^{n-1-i+j}=\sum_{j=0}^{n-1}b_j\F^j+\sum_{j=0}^{i}b_j\sigma^j(c)\F^{n-1-i+j},$$
    that is $$\sum_{j=0}^{n-1}b_j\F^j=-\sum_{j=0}^{i}b_j\sigma^j(c)\F^{n-1-i+j}.$$
    Multiplying recursively on the right by $\F^{n-l}$ for $l=1,\dots,n-1$ one obtains $b_j\F^{n-1}=0$ and since $b_j$ is a Teichm\"uller lift, this implies that $b_j=0$ for all $j=0,\dots,n-1.$ Therefore $\varphi$ is injective and since both the $E$-modules $E/\left(E(\V-\F^i-a\F^{n-1})+E\F^n\right)$ and $E/\left(E(\V-\F^i)+E\F^n\right)$ have length $n$ then $\varphi$ is an isomorphism.
    Let us conclude the proof showing that $$E/\left(E(\V-a\F^{n-1})+E\F^n\right)\simeq E/\left(E(\V-\F^{n-1})+E\F^n\right)$$ when $a\in W(k)^\times.$ Notice that in both $E$-modules we have $p=\F\V=\V\F=0$ since $\F^n=0$ and $\V=a\F^{n-1}$ or $\V=\F^{n-1}$. We can then assume that  $a=[\overline{a}]$ is the Teichm\"{u}ller lift of some element $\overline{a}\in k$. 
    We define the morphism $\psi$ sending $1$ to the Teichm\"uller lift $b=[\overline{b}]$ of a non-zero root $\overline{b}$ of the polynomial $\overline{a}^pX^{p^n}-X$, which exists since $k$ is algebraically closed. Remark that $\psi$ is well-defined since $$\V-a\F^{n-1}\mapsto(\V-a\F^{n-1})b=\sigma^{-1}(b)\V-a\sigma^{n-1}(b)\F^{n-1}=(\sigma^{-1}(b)-a\sigma^{n-1}(b))\F^{n-1}=0.$$ In fact, once more $\V\F^{n-1}=0$ and thus it is enough to verify that $\delta_0=0$ where $\delta:= \sigma^{-1}(b)-a\sigma^{n-1}(b)$, which is true since $$\delta_0=\left(\overline{b}-\overline{a}^p\overline{b}^{p^n}\right)^{1/p}=0.$$
    Finally, $\psi$ is surjective, since $$b^{-1}=\left[\overline{b}^{-1}\right]\mapsto 1$$ and thus an isomorphism. 
\end{proof}

We are now ready to prove Theorem \ref{classification}.

\begin{proof}[Proof of Theorem \ref{classification}]
    We argue by induction on $n.$ For $n=1$ the only infinitesimal commutative unipotent group scheme of order $p$ with one-dimensional Lie algebra is $\alpha_p$ and its Dieudonn\'e module is $E/(E\V+E\F).$ Suppose now that the statement is true for $n-1$ and let $G$ be an infinitesimal commutative unipotent $k$-group scheme of order $p^n$ with one-dimensional Lie algebra. Consider then the short exact sequence $$0\rightarrow\ker(F_G^{n-1})\rightarrow G\rightarrow F_G^{n-1}(G)\simeq\alpha_p\rightarrow0.$$ Then $\ker(F_G^{n-1})$ is a subgroup scheme of $G$ of order $p^{n-1}$ so by inductive hypothesis $$\ker(F_G^{n-1})=W_{n-1}^{n-1}[V-F^i]$$ for some $i=1,\dots,n-1.$ Equivalently we have a surjection \begin{equation}\label{surjection}
        M(G)\twoheadrightarrow M(G)/M(G)\F^{n-1}\simeq E/\left(E(\V-\F^i)+E\F^{n-1}\right)
    \end{equation} where $M(G)$ is the Dieudonn\'e module corresponding to $G$ and we know that $M(G)$ is an $E$-module of length $n$ (the isomorphism is given by Lemma \ref{induction}). The kernel of this surjection is thus $M(G)\F^{n-1}=k\F^{n-1}$, where the equality is due to the fact that $\F^{n}=0$. 
    Therefore, by (\ref{surjection}), it holds $\V-\F^i=a\F^{n-1}$ in $M(G)$ for some $a\in W(k)$ (and we can take $a=[\overline{a}]$ for some $\overline{a}\in k$). As a consequence we have that the surjection $$E/(E\V^n+E\F^n)\twoheadrightarrow M(G)$$ coming from the fact that $G\subseteq W_n^n$ (by \cite[Lemma 2.28]{gouthier2023infinitesimal} together with the last statement of Lemma \ref{infgrps}) factors via $$E/\left(E(\V-\F^i-a\F^{n-1})+E\F^n\right)\twoheadrightarrow M(G).$$ Finally, both the $E$-modules $M(G)$ and $E/\left(E(\V-\F^i-a\F^{n-1})+E\F^n\right)$ have length $n$ and thus $$M(G)\simeq E/\left(E(\V-\F^i-a\F^{n-1})+E\F^n\right).$$ The statement follows by Proposition \ref{Emodulesisom}. 
\end{proof}

\begin{rmk}
    Notice that as a consequence of the proof of Theorem \ref{classification} we have that, for $i=1,\dots,n-2$, $W_n^n[V-F^i]$ is the only infinitesimal commutative unipotent group scheme of order $p^n$ and with one-dimensional Lie algebra that is an extension of $\alpha_p$ by $W_{n-1}^{n-1}[V-F^i]$. On the other hand, there are exactly two infinitesimal commutative unipotent group schemes of order $p^n$ and with one-dimensional Lie algebra that are an extension of $\alpha_p$ by $\alpha_{p^{n-1}}$:  $W_n^n[V-F^{n-1}]$ and $\alpha_{p^n}$.
\end{rmk}

As a direct consequence of Theorem \ref{classification} we can describe all infinitesimal commutative unipotent $k$-group schemes whose Cartier dual have one-dimensional Lie algebra, over an algebraically closed field.

\begin{corollary}\label{cor: duality}
    Let $k$ be algebraically closed. For any $n\geq 1$, there are exactly $n$ non-isomorphic infinitesimal commutative unipotent $k$-group schemes $G$ of order $p^n$ such that $\dim_k(\Lie(G^\vee))=1$. They are the group schemes of the form $$W_n^n[F-V^i]:= \ker(F-V^i\colon W_n^n\rightarrow W_n^n)$$ for some $i=1,\dots,n$ and coincide with the Cartier duals of the $k$-group schemes classified by Theorem \ref{classification}. 
\end{corollary}

\begin{proof}
    By the assumption and applying Theorem \ref{classification}, $G^\vee=W_n^n[V-F^i]$ for some $i=1,\dots,n$. What we need to do is prove that $G=\left(G^\vee\right)^\vee=W_n^n[F-V^i]$. By \cite[Lemma 2.28]{gouthier2023infinitesimal}, since $G^\vee$ has one-dimensional Lie algebra, $G\subseteq W_n^n$. Moreover, by the duality of Frobenius and Verschiebung (see Remark \ref{rmk: Frob and Ver duality}) we deduce that $G\subseteq W_n^n[F-V^i]$. We may now conclude by arguing that the two group schemes have equal order.
\end{proof}

\section{Infinitesimal subgroup schemes of supersingular elliptic curves}

The following proposition describes all infinitesimal subgroup schemes of any supersingular elliptic curve over an algebraically closed field. We recover in particular all their $p^n$-torsions as well: these are already well known and can be deduced, for example, from \cite[Section 5]{deJongOort} together with \cite[Theorem 1.2]{Oortpdiv}. On the other hand, to the author's knowledge, there was still some confusion regarding all (higher) Frobenius kernels of supersingular elliptic curves and Proposition \ref{ellcurves} clarifies this.

\begin{prop}\label{ellcurves}
    Let $k$ be algebraically closed and $E/k$ be a supersingular elliptic curve. Then, for every $n\geq1$ its $p^n$-torsion is $E[p^n]=W_{2n}^{2n}[V-F].$ More precisely, for every $m\geq1$ its $m$-th Frobenius kernel is $E[F^m]=W_m^m[F-V]$ and $E[p^n]$ coincides with $E[F^{2n}]$.
\end{prop}

\begin{proof}
    The $p^n$-torsion $E[p^n]$ of a supersingular elliptic curve $E$ over an algebraically closed field of characteristic $p>0$ is a finite commutative unipotent self-dual $k$-group scheme of order $p^{2n}$ with one-dimensional Lie algebra (see \cite[III.15, Theorem 1]{MumfordAbVar}). By Theorem \ref{classification} and Corollary \ref{cor: duality}, the only such $k$-group scheme is $W_{2n}^{2n}[F-V].$ Let us now consider the Frobenius kernels. Clearly $E[F^m]$ is a $k$-subgroup scheme of order $p^m$ of $E[p^m]$. More precisely, $$E[F^m]=E[p^m]\cap W_{2m}^m=W_m^m[F-V].$$ These can also be observed at the level of $k$-Hopf algebras: indeed $$E[p^m]=\spec\left(k[T_0,\dots,T_{2m-1}]/(T_0^p,T_1^p-T_0,\dots,T_{2m-1}^p-T_{2m-2})\right)$$ and $$E[F^m]=\spec\left(k[T_0,\dots,T_{2m-1}]/(T_0^p,T_1^p-T_0,\dots,T_{2m-1}^p-T_{2m-2},T_0^{p^m},\dots,T_{2m-1}^{p^m})\right)$$ with Hopf algebra structure given by the comultiplication of Witt vectors in both cases. Now, looking at $E[F^m]$ we see that $T_0=T_1^p=T_2^{p^2}=\dots=T_m^{p^m}=0$, $T_1=\dots=T_{m+1}^{p^m}=0,\dots, T_{m-1}=\dots=T_{2m-1}^{p^m}=0$, that is $$E[F^m]=\spec\left(k[T_{m},\dots,T_{2m-1}]/(T_{m}^p,T_{m+1}^p-T_{m},\dots,T_{2m-1}^p-T_{2m-2})\right)$$$$=W_m^m[F-V].$$
So we notice that for $m=2n$ even, it holds that $E[F^{2n}]$ coincides with $E[p^{n}]$.
\end{proof}

\begin{rmk}
Let us point out that, as a consequence of Theorem \ref{classification} and Proposition \ref{ellcurves}, the Frobenius kernels of supersingular elliptic curves are exactly the infinitesimal commutative unipotent self-dual $k$-group schemes with one-dimensional Lie algebra. In addition, these are exactly all the infinitesimal subgroup schemes $G$ of a supersingular elliptic curve $E$ over an algebraically closed field. Indeed, suppose that such a $G$ has height $n$, then $G\subseteq E[F^n]$ and by order reasons they must coincide. 
\end{rmk}

Notice that if an infinitesimal commutative unipotent $k$-group scheme $G$ with $n$-dimensional Lie algebra can be embedded in a smooth connected $n$-dimensional algebraic group $\mathcal{G}$, then $G$ acts freely on it (by multiplication). Brion \cite[\textsection1]{brion2022actions} asked if there are examples of generically free rational actions on curves of infinitesimal commutative unipotent group schemes that are not subgroup schemes of a smooth connected one-dimensional algebraic group. Recall that if $\mathcal{G}$ is a smooth connected one-dimensional $k$-algebraic group, then, either $\mathcal{G}$ is affine and $\mathcal{G}_{\overline{k}}\simeq\G_{m,\overline{k}}$, or $\mathcal{G}_{\overline{k}}\simeq\G_{a,\overline{k}}$, or $\mathcal{G}$ is an elliptic curve. The following proposition explains that if $k$ is algebraically closed, very few infinitesimal commutative unipotent $k$-group schemes with one-dimensional Lie algebra are contained in smooth connected one-dimensional algebraic groups. All the others are examples of infinitesimal group schemes that act rationally and generically freely on any curve (by \cite[Theorem 1.4]{gouthier2023infinitesimal}), but are not subgroup schemes of a smooth connected one-dimensional algebraic group.

\begin{prop}\label{guciinellcur}
    Let $k$ be algebraically closed and $G$ be an infinitesimal commutative unipotent $k$-group scheme with one-dimensional Lie algebra. Then $G$ is contained in a smooth connected one-dimensional algebraic group if and only if either $G\simeq\alpha_{p^n}\subseteq\G_a$ or $G$ is an infinitesimal subgroup scheme of a supersingular elliptic curve. 
\end{prop}

\begin{proof}
  By Theorem \ref{classification}, $G\simeq W_n^n[V-F^i]$ for some $n\geq1$ and $i=1,\dots,n$.
  Let us start by considering the cases $i=1,n$ for any $n\geq 1$. For $i=1$ we have $G\simeq W_n^n[F-V]$ which is the $n$-th Frobenius kernel of a supersingular elliptic curve (Proposition \ref{ellcurves}). For $i=n$, then $G\simeq W_1^n=\alpha_{p^n}\subseteq\G_a$.
  
  Let us show that if $n>2$ and $i=2,\dots,n-1$ then $G$ is not contained in a smooth connected one-dimensional algebraic group. 
    Under these assumptions, clearly $G$ is neither a subgroup of $\G_m$ (since $G$ is unipotent) nor of $\G_a$ (since $V_G\neq0$). Therefore, if $G$ is a subgroup of a smooth connected one-dimensional algebraic group, then it is a subgroup of an elliptic curve $E$. In particular, since $G$ is infinitesimal unipotent, $E$ has to be supersingular and $G$ is one of its infinitesimal subgroup schemes. This gives a contradiction since, by Proposition \ref{ellcurves}, those are obtained exactly for $i=1$. 
\end{proof}

We conclude by addressing the question of whether all infinitesimal unipotent group schemes with one-dimensional Lie algebra are commutative, raised by both Fakhruddin and Brion (see \cite[Remark 2.10]{Fakhruddin} and \cite[\textsection1]{brion2022actions}). The following example shows that this is not the case.

\begin{example}\label{noncommutative}
    Consider the infinitesimal unipotent non-commutative $k$-group scheme $G=\spec(A)$ where $$A=k[T_0,T_1]/\left(T_0^{p^n}, T_1^p-T_0\right)$$ with $n\geq2$ an integer and comultiplication given by $$\Delta(T_0)=T_0\otimes1+1\otimes T_0$$ and $$\Delta(T_1)=T_1\otimes1+1\otimes T_1+T_0^{p^{n-1}}\otimes T_0^{p^{n-2}}.$$ In this case $$A^\vee=$$$$k\langle U_0\dots,U_{n}\rangle/(U_0^p,\dots,U_n^p,U_iU_j-U_jU_i,U_nU_{n-1}-U_{n-1}U_n-U_0)_{i,j=0,\dots n,(i,j),(j,i)\neq(n,n-1)}$$ where $U_0(T_1)=1$ and $U_i(T_0^{p^{i-1}})=1$ and zero elsewhere. The Hopf algebra $A^\vee$ is non-commutative: the only non-commutative relation is given by $U_nU_{n-1}-U_{n-1}U_n=U_0$, while its comultiplication is defined on the $U_i$'s as for the Witt vectors (notice that this makes sense since $U_0,\dots,U_{n-1}$ commute). 
    These examples arise as closed subgroup schemes of non-commutative extensions of $\G_a$ by itself (see \cite[II.\textsection3, 4]{DG}) and there are many of them. Other examples can be found in \cite{gouthier2024unexpected}. In particular $G$ is a subgroup scheme of $\mathrm{PGL}_{2,k}$ when $k$ has characteristic $2$ and $n=3$ (see \cite[Theorem 1.1]{gouthier2024unexpected}).
\end{example}

\bibliographystyle{alpha}
\bibliography{bib}

@misc{brion2022actions,
      title={Actions of finite group schemes on curves}, 
      author={Michel Brion},
      year={2022},
      eprint={2207.08209},
      archivePrefix={arXiv},
      primaryClass={math.AG}
}

@book {Montgomery,
    AUTHOR = {Montgomery, Susan},
     TITLE = {Hopf algebras and their actions on rings},
    SERIES = {CBMS Regional Conference Series in Mathematics},
    VOLUME = {82},
 PUBLISHER = {Published for the Conference Board of the Mathematical
              Sciences, Washington, DC; by the American Mathematical
              Society, Providence, RI},
      YEAR = {1993},
     PAGES = {xiv+238},
      ISBN = {0-8218-0738-2},
   MRCLASS = {16W30},
  MRNUMBER = {1243637},
MRREVIEWER = {E. J. Taft},
       DOI = {10.1090/cbms/082},
       URL = {https://doi.org/10.1090/cbms/082},
}

@book {MumfordAbVar,
    AUTHOR = {Mumford, David},
     TITLE = {Abelian varieties},
    SERIES = {Tata Institute of Fundamental Research Studies in Mathematics},
    VOLUME = {5},
      NOTE = {With appendices by C. P. Ramanujam and Yuri Manin,
              Corrected reprint of the second (1974) edition},
 PUBLISHER = {Tata Institute of Fundamental Research, Bombay; by Hindustan
              Book Agency, New Delhi},
      YEAR = {2008},
      ISBN = {978-81-85931-86-9; 81-85931-86-0},
   MRCLASS = {14Kxx},
  MRNUMBER = {2514037},
}

@book {DG,
    AUTHOR = {Demazure, Michel and Gabriel, Pierre},
     TITLE = {Groupes alg\'{e}briques. {T}ome {I}: {G}\'{e}om\'{e}trie
              alg\'{e}brique,
 g\'{e}n\'{e}ralit\'{e}s, groupes
              commutatifs.},
      NOTE = {Avec un appendice {\itshape Corps de classes local} par Michiel
              Hazewinkel.},
 PUBLISHER = {Masson \& Cie, \'{E}diteurs, Paris and North-Holland Publishing
              Co., Amsterdam},
      YEAR = {1970},
     PAGES = {xxvi+700},
   MRCLASS = {14L15 (20G35)},
  MRNUMBER = {302656},
MRREVIEWER = {J.-E.\ Bertin},
}

@book {Milne,
    AUTHOR = {Milne, J. S.},
     TITLE = {Algebraic groups},
    SERIES = {Cambridge Studies in Advanced Mathematics},
    VOLUME = {170},
      NOTE = {The theory of group schemes of finite type over a field},
 PUBLISHER = {Cambridge University Press, Cambridge},
      YEAR = {2017},
     PAGES = {xvi+644},
      ISBN = {978-1-107-16748-3},
   MRCLASS = {14L15 (14-01 17B45 20-01 20G15)},
  MRNUMBER = {3729270},
MRREVIEWER = {Boris\ \`E.\ Kunyavski\u{\i}},
       DOI = {10.1017/9781316711736},
       URL = {https://doi.org/10.1017/9781316711736},
}

@article {NWW,
    AUTHOR = {Nguyen, Van C. and Wang, Linhong and Wang, Xingting},
     TITLE = {Classification of connected {H}opf algebras of dimension
              {$p^3$} {I}},
   JOURNAL = {J. Algebra},
  FJOURNAL = {Journal of Algebra},
    VOLUME = {424},
      YEAR = {2015},
     PAGES = {473--505},
      ISSN = {0021-8693,1090-266X},
   MRCLASS = {16T05},
  MRNUMBER = {3293230},
MRREVIEWER = {Shixun\ Lin},
       DOI = {10.1016/j.jalgebra.2014.09.022},
       URL = {https://doi.org/10.1016/j.jalgebra.2014.09.022},
}

@book {Oort,
    AUTHOR = {Oort, F.},
     TITLE = {Commutative group schemes.},
    SERIES = {},
 PUBLISHER = {Springer-Verlag, Berlin-New York},
      YEAR = {1966},
     PAGES = {vi+133 pp. (not consecutively paged)},
   MRCLASS = {14.50 (18.00)},
  MRNUMBER = {213365},
MRREVIEWER = {S.\ S.\ Shatz},
}

@article {Fakhruddin,
    AUTHOR = {Fakhruddin, Najmuddin},
     TITLE = {Finite group schemes of essential dimension one},
   JOURNAL = {Doc. Math.},
  FJOURNAL = {Documenta Mathematica},
    VOLUME = {25},
      YEAR = {2020},
     PAGES = {55--64},
      ISSN = {1431-0635,1431-0643},
   MRCLASS = {14L15 (12G05 14G17)},
  MRNUMBER = {4077548},
MRREVIEWER = {B.\ Sury},
       DOI = {10.4171/DM/737},
}

@article{Oortpdiv,
 ISSN = {0003486X},
 URL = {http://www.jstor.org/stable/3597323},
 author = {Frans Oort},
 journal = {Annals of Mathematics},
 number = {2},
 pages = {1021--1036},
 publisher = {Annals of Mathematics},
 title = {Minimal p-Divisible Groups},
 urldate = {2024-12-17},
 volume = {161},
 year = {2005}
}

@article {deJongOort,
    AUTHOR = {de Jong, A. J. and Oort, F.},
     TITLE = {Purity of the stratification by {N}ewton polygons},
   JOURNAL = {J. Amer. Math. Soc.},
  FJOURNAL = {Journal of the American Mathematical Society},
    VOLUME = {13},
      YEAR = {2000},
    NUMBER = {1},
     PAGES = {209--241},
      ISSN = {0894-0347,1088-6834},
   MRCLASS = {14L05 (14B05)},
  MRNUMBER = {1703336},
MRREVIEWER = {Nobuo\ Tsuzuki},
       DOI = {10.1090/S0894-0347-99-00322-7},
       URL = {https://doi.org/10.1090/S0894-0347-99-00322-7},
}

@online {Pink,
    AUTHOR = {Pink, Richard},
     TITLE = {Finite group schemes},
     NOTE = {Lecture course in WS 2004/05, ETH Z\"urich},
      YEAR = {2005}}

@article {gouthier2023infinitesimal,
    AUTHOR = {Gouthier, Bianca},
     TITLE = {Infinitesimal rational actions},
   JOURNAL = {Trans. Amer. Math. Soc.},
  FJOURNAL = {Transactions of the American Mathematical Society},
    VOLUME = {378},
      YEAR = {2025},
    NUMBER = {9},
     PAGES = {6485--6533},
      ISSN = {0002-9947},
   MRCLASS = {14L15 (14L17 14L30 16T05 16T10)},
  MRNUMBER = {4941122},
       DOI = {10.1090/tran/9413},
       URL = {https://doi.org/10.1090/tran/9413},
}

@misc{gouthier2024unexpected,
      title={Unexpected subgroup schemes of {$\mathrm{PGL}_{2,k}$} in characteristic 2}, 
      author={Bianca Gouthier and Dajano Tossici},
      year={2024},
      eprint={2403.09469},
      archivePrefix={arXiv},
      primaryClass={math.AG}
}

@incollection {pries2006short,
    AUTHOR = {Pries, Rachel},
     TITLE = {A short guide to {$p$}-torsion of abelian varieties in
              characteristic {$p$}},
 BOOKTITLE = {Computational arithmetic geometry},
    SERIES = {Contemp. Math.},
    VOLUME = {463},
     PAGES = {121--129},
 PUBLISHER = {Amer. Math. Soc., Providence, RI},
      YEAR = {2008},
      ISBN = {978-0-8218-4320-8},
   MRCLASS = {11G15 (11G10 14K10)},
  MRNUMBER = {2459994},
MRREVIEWER = {Adrian\ Vasiu},
       DOI = {10.1090/conm/463/09051},
       URL = {https://doi.org/10.1090/conm/463/09051},
}

\end{document}